\documentclass{amsart}
\usepackage{amsfonts}
\usepackage{amsmath}
\usepackage{amssymb}
\usepackage{amsthm}

\newtheorem{theorem}{Theorem}

\newtheorem{lemma}{Lemma}

\begin{document}

\title[Isometric embeddings of dual polar graphs in Grassmann graphs]
{Isometric embeddings of dual polar graphs in Grassmann graphs over finite fields}
\author{Mark Pankov}
\subjclass[2000]{51A50, 51E24}
\keywords{Grassmann graph, dual polar graph, isometric embedding}
\address{Department of Mathematics and Computer Science, 
University of Warmia and Mazury,
S{\l}oneczna 54, Olsztyn, Poland}
\email{pankov@matman.uwm.edu.pl}

\maketitle

\begin{abstract}
We consider the Grassmann graphs and dual polar graphs over
the same finite field and show that, up to graph automorphism, for every dual polar graph 
there is the unique isometric embedding in the corresponding Grassmann graph.
\end{abstract}

\section{Introduction}
Grassmann graphs and polar Grassmann graphs (not necessarily over finite fields) 
are interesting for many reasons 
\cite{BC-book,D-book,Pankov-book1,Pankov-book2,Pasini-book, Shult-book}.
For example, they are closely related to buildings of classical types \cite{Tits}.
Also, Grassmann graphs and dual polar graphs (over finite fields)
are classical examples of distance regular graphs \cite{BCN-book}.
Embeddings (not necessarily isomorphic in some cases) 
of Grassmann graphs and polar Grassmann graphs 
(over division rings) are investigated in 
\cite{Pankov-paper1,Pankov-book2,KP,Pankov-paper2}.

All dual polar graphs defined by sesquilinear, quadratic and pseudo-quadratic forms
are naturally isometrically embedded in the corresponding Grassmann graphs.
In this short note we consider the Grassmann graphs and dual polar graphs over
the {\it same finite} field. 
We show that, up to graph automorphism, for every dual polar graph 
there is the unique isometric embedding in the corresponding Grassmann graph. 
This statement is related to the problem formulated in \cite{KMP}.

The author's interest is the general case when 
Grassmann graphs and dual polar graphs are 
related to different not necessarily finite division rings. 
To describe isometric embeddings and  get a result in spirit of \cite[Chapter 3]{Pankov-book2},
we need semilinear embeddings of special type. 
This is a topic for a more detailed research.

\section{Result}
Let $V$ be an $n$-dimensional  vector space over a division ring. 
Denote by ${\mathcal G}_{k}(V)$ the Grassmannian formed by 
$k$-dimensional subspaces of $V$.
The corresponding {\it Grassmann graph} $\Gamma_{k}(V)$ is 
the graph whose vertex set is ${\mathcal G}_{k}(V)$
and two $k$-dimensional subspaces are adjacent vertices of this graph if their intersection is $(k-1)$-dimensional.
If $k=1,n-1$ then any two distinct vertices of $\Gamma_{k}(V)$ are adjacent and we will alway suppose that
$1<k<n-1$. Also, we can suppose that $2k\le n$, since the Grassmann graphs $\Gamma_{k}(V)$ and $\Gamma_{n-k}(V^{*})$ are 
isomorphic.

We write $\Pi_{V}$ for the projective space associated to $V$ 
(the points are the $1$-dimensional subspaces and the lines are defined by the $2$-dimensional subspaces).
Let $\Pi=({\mathcal P},{\mathcal L})$ be a rank $m$ polar space,
see \cite{BC-book, Ueberberg} for the precise definition.
Also, we suppose that the polar space $\Pi$ is embedded in 
the projective space $\Pi_{V}$, i.e. 
the lines of $\Pi$ are lines of $\Pi_{V}$. 
Then $2m\le n$ and all maximal singular subspaces of $\Pi$ can be identified  with some $m$-dimensional subspaces of $V$.
The set of all such $m$-dimensional subspaces is denoted by ${\mathcal G}(\Pi)$.
The corresponding {\it dual polar graph} $\Gamma(\Pi)$ 
is the restriction of the Grassmann graph $\Gamma_{m}(V)$
to the set of maximal singular subspaces ${\mathcal G}(\Pi)$.

The existence of isometric embeddings of $\Gamma(\Pi)$ in $\Gamma_{k}(V)$ 
implies that $m\le k$, i.e.
the diameter of the dual polar graph is not greater than the diameter of the Grassmann graph.

\begin{theorem}
If $m\le k$ and $V$ is a vector space over a finite field then, 
up to automorphism of the Grassmann graph $\Gamma_{k}(V)$, 
there is the unique isometric embedding of $\Gamma(\Pi)$ in $\Gamma_{k}(V)$.
\end{theorem}

\section{Sketch of proof}

Let $U$ be a subspace of $V$ whose dimension is less than $k$. 
Denote by $[U\rangle_{k}$ the set of all $k$-dimensional subspaces containing $U$.
If $U$ is a singular subspace for $\Pi$ then 
$$[U\rangle:={\mathcal G}(\Pi)\cap [U\rangle_{m}$$
is the set of all maximal singular subspaces containing $U$.

Let $f:{\mathcal G}(\Pi)\to {\mathcal G}_{k}(V)$ 
be an isometric embedding of $\Gamma(\Pi)$ in $\Gamma_{k}(V)$,
i.e. an injection preserving the distance between vertices.

\begin{lemma}
The image of $f$ is contained in $[U\rangle_{k}$,
where $U$ is a $(k-m)$-dimensional subspace.
\end{lemma}

\begin{proof}
A simple modification of the proof \cite[Lemma 6]{Pankov-paper1}.
\end{proof}

By Lemma 1, our isometric embedding 
can be considered as an isometric embedding $g$
of $\Gamma(\Pi)$ in $\Gamma_{m}(W)$, where $W=V/U$.

Let $P$ be a $1$-dimensional subspace of $V$
which is a point of the polar space $\Pi$.
All lines of $\Pi$ containing $P$ form a polar space of rank $m-1$
and the associated dual polar graph is 
the restriction of $\Gamma(\Pi)$ to the set $[P\rangle$.
Lemma 1 implies the existence of a $1$-dimensional subspace $q(P)\subset W$
such that 
$$g([P\rangle)\subset [q(P)\rangle_{m}.$$
In other words, our isometric embedding induces a certain mapping 
$q:{\mathcal P}\to {\mathcal G}_{1}(W)$.
The mapping $q$ is injective (this follows from the fact that $g$ is an isometric embedding).
Next, we establish the following: if $P$ and $Q$ are collinear points of $\Pi$
then
$$g([P+Q\rangle)\subset [q(P)+q(Q)\rangle_{m}$$ 
which implies that $q$ sends lines of $\Pi$ to subsets in lines of $\Pi_{W}$
(our objects are not necessarily over the same finite field).

However, if $V$ is over a finite field
then $q$ sends every line of $\Pi$ to a line of $\Pi_{W}$. 
The restriction of $q$ to every maximal singular subspace of $\Pi$
is a collineation to a certain projective space.
By the Fundamental Theorem of Projective Geometry,
this restriction is induced by a semilinear isomorphism 
between the corresponding vector spaces. 
Using \cite[Section III.3, f)]{D-book}, 
we extend $q$ to a collineation of $\Pi_{V'}$ to $\Pi_{W'}$,
where $V'$ is the minimal subspace of $V$ containing $\Pi$
and $W'$ is a subspace of $W$
\footnote{This arguments do not work if Grassmann graphs and dual polar graphs are 
related to different not necessarily finite division rings}.
This gives the claim.

\end{document}